\documentclass[11pt]{article}
\usepackage{latexsym,color,amsmath,amsthm,amssymb,amscd,amsfonts,mathrsfs,array,xfrac}
\usepackage{tikz}
\usetikzlibrary{arrows,calc}
\tikzset{
>=stealth',
help lines/.style={dashed, thick}, axis/.style={<->}, important
line/.style={thick}, connection/.style={thick, dotted}, }

\def\conv{{\rm Conv}}
\def\RR{\mathbb{R}}
\def\R{\mathbb{R}}
\def\inrad{\mathop\mathrm{inrad}\nolimits}
\newtheorem{lemma}{Lemma}[section]

\newtheorem{remark}[lemma]{Remark}
\newtheorem{theorem}[lemma]{Theorem}
\newtheorem{definition}[lemma]{Definition}

\newtheorem{conj}[lemma]{Conjecture}

\newtheorem*{remark*}{Remark}

\newtheorem*{def*}{Definition}
\newtheorem*{prop*}{Proposition}

\makeatletter \@addtoreset {equation}{section}

\renewcommand\theequation
  {\ifnum \c@subsection>\z@ \arabic{section}.\arabic{subsection}.\arabic{equation}
  \else \arabic{section}.\arabic{equation} \fi}
\makeatother
\begin{document}
\title{From Symplectic Measurements to the Mahler Conjecture} 
\author{ \ Shiri Artstein-Avidan, \  Roman  Karasev,    \ Yaron Ostrover}
\date{}
\maketitle
\begin{abstract}
In this note we link symplectic and convex geometry by relating two 
seemingly different open conjectures: a symplectic isoperimetric-type inequality for
convex domains, and  Mahler's conjecture on the volume
product of centrally symmetric convex bodies. More precisely, we show that if for convex bodies of fixed volume  in the classical phase space the Hofer--Zehnder capacity is maximized by the Euclidean ball, then a hypercube is a minimizer for the volume product  among  centrally symmetric convex bodies.
\end{abstract}

\section{Introduction and Main Results}

The purpose of this note is to relate a symplectic isoperimetric-type
conjecture for symplectic capacities of convex domains in the classical phase space with the renowned
Mahler conjecture regarding the volume product of symmetric convex bodies.
The main ingredient in the proof, which is of independent interest in the context of symplectic geometry, is the fact that in a centrally symmetric and strictly convex body $K\subseteq \RR^n$
the shortest (measured with respect to $\| \cdot \|_K$) periodic $K^{\circ}$-billiard trajectory  is a $2$-bouncing trajectory. In particular, this implies that for any  centrally symmetric convex body  $K$ the Hofer--Zendher capacity of the configuration $K\times K^{\circ}$ in the classical phase space $\RR^{2n}$ equals $4$.

Before we state our main results more precisely, we first recall some relevant
background and definitions from convex and symplectic geometry.

\subsection{Mahler Conjecture}
Let $(X,\| \cdot \|)$ be an $n$-dimensional normed space and let
$(X^*,\| \cdot \|^*)$  be its dual space. The product space $X
\times X^*$ carries a canonical symplectic structure, given by the
skew-symmetric bilinear form $\omega \bigl ( (x,\xi),(x',\xi') \bigr
) = \xi(x')-\xi'(x)$, and a canonical volume form, the {\it
Liouville} volume, given by $ \omega^n/n!$. A fundamental question
in the field of convex geometry, raised by Mahler in~\cite{Ma}, is to find  upper and lower bounds for the
Liouville volume of $B \times B^{\circ} \subset X \times X^*$, where
$B$ and $B^{\circ}$ are the unit balls of $X$ and $X^*$
respectively. In what follows we shall denote this volume by
$\nu(X)$. 
Note that $\nu(X)$ is an affine invariant of $X$ i.e., it
is invariant under invertible linear transformations. We remark that
in the context of convex geometry $\nu(X)$ is also known as the
{\it Mahler volume} or the {\it volume product} of $X$.

The Blaschke--Santal\'o inequality asserts that the maximum of
$\nu(X)$ is attained if and only if $X$ is a Euclidean space. This
was proved by Blaschke~\cite{Bl} for dimensions two and three, and
generalized by Santal\'o~\cite{Sa} for higher dimensions. 
The following sharp lower bound for $\nu(X)$ was
conjectured by Mahler~\cite{Ma} in 1939:

\noindent {\bf Mahler Conjecture:} \label{Mahler-Conj} For an
$n$-dimensional normed space $X$ one has $\nu(X) \geq 4^n/n!$ 

The conjecture has been verified by Mahler~\cite{Ma} in the
two-dimensional case. In higher dimensions it is proved only in
some very special cases, namely, when the unit ball of $X$ is a
zonoid~\cite{GMR,R1}, when $X$ has a 1-unconditional
basis~\cite{ME1,SR,R2}, and when the unit ball of
$X$ is sufficiently close to the unit cube in the Banach--Mazur
distance~\cite{NPRZ}. 

The first major breakthrough towards answering Mahler's conjecture was a result 
due to Bourgain and Milman~\cite{BM}, who  used sophisticated tools from functional analysis to
show that the conjecture holds asymptotically, i.e., up to a factor
$\gamma^n$, where $\gamma$ is a universal constant. 
This result has been re-proved later on, by entirely different methods, by Kuperberg~\cite{Ku}, using differential geometry, 
and independently by Nazarov~\cite{Naz},  using the theory of functions of several complex variables. A new proof using simpler asymptotic geometric analysis tools has been recently discovered by Giannopoulos, Paouris, and Vritsiou~\cite{GPV}. 
The best known constant nowadays, $\gamma = \pi/4$, is due to Kuperberg~\cite{Ku}.
Despite great efforts to deal with the general case, a proof of Mahler's conjecture has been insistently elusive thus far, and is currently the subject of intensive research efforts.
We remark that in
contrast with the above mentioned Blaschke--Santal\'o inequality, the equality case
in Mahler's conjecture, which is obtained for example for the space
$l^n_{\infty}$ of bounded sequences with the standard maximum norm,
is not unique.

\subsection{Symplectic Capacities} \label{subsection-capacities}

Consider the $2n$-dimensional Euclidean space ${\mathbb R}^{2n}  =
{\mathbb R}^n_q \times {\mathbb R}^n_p$,  equipped with the linear
coordinates $(q_1,\ldots, q_n,p_1,\ldots,p_n)$,  the standard
symplectic structure $\omega_{\rm st} = \sum_i dq_i \wedge dp_i$, and  the
standard inner product $g_{\rm st} = \langle \cdot, \cdot \rangle$.  In what follows we use this inner product to identify 
the tangent space $T_x {\mathbb R}^{2n}$, at a point $x \in {\mathbb R}^{2n}$, with the cotangent space $T^*_x {\mathbb R}^{2n}$ at the same point in the usual way.
Note
that under the identification of ${\mathbb R}^{2n}$ with
${\mathbb C}^n$, these two structures are the real and the imaginary
parts of the standard Hermitian inner product in ${\mathbb  C}^n$,
and $\omega_{\rm st}(v,Jv) =  \langle v, v \rangle$, where $J$ is the
standard complex structure on ${\mathbb R}^{2n} \simeq {\mathbb
C}^n$. Recall that a symplectomorphism of ${\mathbb R}^{2n}$ is a
diffeomorphism which preserves the symplectic structure i.e., $\psi
\in {\rm Diff}({\mathbb R}^{2n})$ such that $\psi^* \omega =
\omega$. 

A  fundamental result in symplectic geometry states that 
symplectic manifolds -- in a sharp contrast to Riemannian manifolds -- have 
no local invariants (except, of course, the dimension). 
The first examples of global symplectic invariants were  introduced by Gromov in his 
seminal paper~\cite{G}, where he used  pseudoholomorphic curves techniques to prove a striking symplectic  rigidity result,  
which is nowadays known as Gromov's ``non-squeezing theorem".  This result paved the way to the introduction of global symplectic 
invariants, called symplectic capacities, which 
roughly speaking measure the symplectic size of sets in ${\mathbb
R}^{2n}$. More precisely, let $B^{2k}(r)$ be the open $2k$-dimensional ball of radius $r$.

\begin{definition} \label{Def-sym-cap}
A symplectic capacity on $({\mathbb R}^{2n},\omega_{\rm st})$ associates
to each  subset $U \subset {\mathbb R}^{2n}$ a number $c(U) \in
[0,\infty]$, such that the following three properties hold:

\noindent (P1) $c(U) \leq c(V)$ for $U \subseteq V$ (monotonicity),

\noindent (P2) $c \big (\psi(U) \big )= |\alpha| \, c(U)$ for  $\psi
\in {\rm Diff} ( {\mathbb R}^{2n} )$ such that $\psi^*\omega_{\rm st} =
\alpha \, \omega_{\rm st}$ (conformality),

\noindent (P3) $c \big (B^{2n}(r) \big ) = c \big (B^2(r) \times
{\mathbb C}^{n-1} \big ) = \pi r^2$ (nontriviality and
normalization).
\end{definition}
Note that the third property disqualifies any volume-related invariant, while the first two imply that for 
$U, V \subset {\mathbb  R}^{2n}$, a necessary condition for the existence of a symplectomorphism $\psi$ 
with $\psi(U) = V$ is that $c(U ) = c(V)$ for any symplectic capacity $c$. 

It is a priori unclear that symplectic capacities exist. 
The above mentioned non-squeezing result naturally leads to the definition of two symplectic 
capacites: the Gromov width, defined by $\underline c(U)=\sup\{\pi r^2 \, | \,  B^{2n}(r) \stackrel{\rm s} \hookrightarrow U \} $; and the 
cylindrical capacity, defined by $\overline c(U) = \inf\{\pi r^2 \, | \, U \stackrel{\rm s} \hookrightarrow Z^{2n}(r) \} $, where $\stackrel{\rm s} \hookrightarrow$ stands for symplectic embedding, and $Z^{2n}(r)=B^2(r) \times
{\mathbb C}^{n-1}$ is the standard symplectic cylinder of radius $r$. 
These two capacities are known to be the largest and the smallest possible symplectic capacities, respectively.

Since Gromov's work, several other symplectic capacities were constructed, such as the Hofer--Zehnder capacity~\cite{HZ},  Ekeland--Hofer capacities~\cite{EH}, the displacement
energy~\cite{H1}, spectral capacities~\cite{FGS,Oh,V2}, and Hutchings' embedded contact
homology capacities~\cite{Hut}.  
Moreover, there has been a great
progress in understanding their properties, interrelations, and
applications to symplectic topology and Hamiltonian dynamics. We note that usually computing these capacities, even for relatively simple sets, is notoriously difficult.
We refer the reader to~\cite{CHLS} for a detailed
survey on the theory of symplectic capacities.

\subsection{A Symplectic Isoperimetric Conjecture}

Let ${\cal K}^{2n}$ be the class of convex domains in ${\mathbb R}^{2n}$.  
The following isoperimetric-type conjecture for symplectic capacities of convex bodies was raised by Viterbo in~\cite{V}.
\begin{conj}(Symplectic isomperimetric conjecture): \label{iso-per-conj}
For any symplectic capacity $c$ and any convex body $\Sigma \in {\cal K}^{2n}$, 
\begin{equation} \label{ineq-vitconj} {\frac {c(\Sigma)} {c(B)}} \leq \Bigl (
{\frac {{\rm Vol}(\Sigma)} {{\rm Vol}(B)}} \Bigr )^{1/n}, \ \ {\rm where}
\ B = B^{2n}(1). \end{equation}
\end{conj}

In other words, the symplectic isoperimetric  conjecture states that among the convex domains in ${\mathbb R}^{2n}$ 
with a given volume, the Euclidean ball has the maximal symplectic capacity. 
The 
conjecture is known to hold for certain classes of convex  bodies,
including ellipsoids and convex Reinhardt domains (see~\cite{Her}).
Moreover, up to a universal  constant,
Conjecture~\ref{iso-per-conj} 
holds for any symplectic capacity. More precisely, the following
theorem was proved in~\cite{AAMO}.
\begin{theorem} \label{up-to-uni-constnat} There is a universal constant $A_0$, such that for
any $n$, any $\Sigma \in {\cal K}({\mathbb R}^{2n})$, and any symplectic
capacity $c$, 
\begin{equation} \label{AAMO-result} {\frac {c(\Sigma)} {c(B)}} \leq A_0 \, \Bigl (
{\frac {{\rm Vol}(\Sigma)} {{\rm Vol}(B)}} \Bigr )^{1/n}, \ \ {\rm where}
\ B = B^{2n}(1). \end{equation}
\end{theorem}
This theorem improves a previous result of Viterbo in~\cite{V},
where inequality~$(\ref{AAMO-result})$ was proved up to a factor
depending linearly on the dimension. 
\begin{remark} {\rm 
It is a long standing open question (see, e.g.,~\cite{Her,Ho,V}) whether 
all symplectic capacities coincide on the class of convex domains in ${\mathbb R}^{2n}$.
Note that an affirmative answer to this question would immediately prove Conjecture~\ref{iso-per-conj} above, as it is not hard to check 
that  inequality~$(\ref{ineq-vitconj})$ trivially holds for the Gromov width capacity.
}
\end{remark}

\subsection{Main Results}
Consider the classical phase space ${\mathbb R}^{2n}  =
{\mathbb R}^n_q \times {\mathbb R}^n_p$  equipped with the standard symplectic structure $\omega_{\rm st}$. 
Let ${\cal K}_{s}({\mathbb R}_q^{n})$ be the class of centrally
symmetric convex bodies in ${\mathbb R}_q^{n}$ i.e.,  bounded
convex domains which are symmetric with respect to the origin and
with non-empty interior. For $K \in {\cal K}_s({\mathbb R}_q^n)$, we
define its polar body $K^{\circ} \in {\cal K}_s({\mathbb R}_p^n)$ to
be $$ K^{\circ} = \{p \in {\mathbb R}^n_p \ | \ p(q) \leq 1, \ \
{\rm for \ every \ } q \in K \}.$$ Here we identified ${\mathbb
R}^n_p$ with the dual space $({\mathbb R}^n_q)^*$. Note that if $K$
is considered to be the unit ball of a certain norm $\| \cdot \|$ on
${\mathbb R}_q^n$, then $K^{\circ}$ can be interpreted as the unit
ball of the dual space ${\mathbb R}_p^n \simeq ({\mathbb R}_q^n)^*$
equipped with the dual norm $\| \cdot \|^*$.
In these notations, letting  ${\rm Vol}$ denote the standard
volume in ${\mathbb R}^{2n} = {\mathbb R}^n_q \times {\mathbb
R}^n_p$, the Mahler conjecture reads:

\begin{conj}[Mahler]\label{conj:mahler}
For every $K \in {\cal K}_s({\mathbb R}_q^n)$, one has ${\rm Vol}(K \times K^{\circ}) \geq
{{4^n}/ {n!}}.$
\end{conj}
Note that, by the continuity of volume and the denseness of smooth bodies in ${\cal K}_s({\mathbb R}_q^n)$ (say, with respect to  the Hausdorff metric), it is enough to prove Conjecture~\ref{conj:mahler} for smooth $K$.
We are finally in a position to state our main result:
\begin{theorem} \label{main-thm}
The symplectic isoperimetric conjecture implies the Mahler conjecture.
\end{theorem}
The proof of Theorem~\ref{main-thm} follows immediatly from the following estimate of the symplectic size 
of the configuration $K \times K^{\circ} \subset {\mathbb R}^n_q \times {\mathbb R}^n_p$, for $K \in {\cal K}_s({\mathbb R}_q^{n})$, which is of independent interest for symplectic geometry.
Let $ c_{_{\rm HZ}}$ denote the Hofer--Zehnder capacity, which we shall define in detail in Section \ref{sec-HZ-Mink-billiards}. It is known  (see, e.g., \textsection 3.5 in~{\cite{HZ}, and Section~\ref{sec-HZ-Mink-billiards} below), that on the class of convex domains in ${\mathbb R}^{2n}$ the capacity $c_{_{\rm HZ}}(\Sigma)$
is given by the minimal action of closed characteristics on the boundary $\partial \Sigma$.

\begin{theorem} \label{thm-about-capacities} For every centrally symmetric convex body $K \in {\cal K}_s({\mathbb R}_q^{n})$, 
$$ c_{_{\rm HZ}}(K \times K^{\circ})   =4.$$
\end{theorem}

In fact, 
Theorem~\ref{thm-about-capacities} can be strengthened to show that for any two centrally symmetric convex bodies 
$K \subset {\mathbb R}^n_q$, and $T  \subset {\mathbb R}^n_p$, one has  
\begin{equation} \label{cap-of-Lag-Prod-Symetric} c_{_{\rm HZ}}(K \times T)   =  \overline c(K \times T) =  4 \inrad_{T^{\circ}}(K), \end{equation} 
where  $\inrad_{T^{\circ}}(K) = \max \{ r> 0 \, | \, rT{^\circ} \subset K \}$. See Remark~\ref{rem-about-K-T}  
in Section~\ref{sec-proof-of-main-result} for the explanation of this fact.

The proof of Theorem~\ref{thm-about-capacities} is based on the relation established in~\cite{AAO1} between the Hofer--Zehnder capacity of 
certain convex Lagrangian products in the classical phase space, and the minimal length of periodic Minkowski billiard trajectories associated with these products. 
The precise details of this relation will be given in Section~\ref{sec-HZ-Mink-billiards} below. We shall prove
\begin{theorem} \label{thm-about-strictly-convex-billiards}
Let $K\subset \RR_q^n$ be a centrally symmetric strictly convex body. The shortest (with respect to $\| \cdot \|_K$) periodic $K^{\circ}$-billiard trajectory in $K$ is attained exclusively by two-bouncing orbits. In particular, the minimal $\|\cdot \|_K$-length of a periodic $K^{\circ}$-billiard trajectory in $K$ is  $4$. 
\end{theorem}

We now turn to showing that  Theorem~\ref{main-thm} follows from Theorem~\ref{thm-about-capacities}.
\begin{proof}[{\bf Proof of Theorem~\ref{main-thm}}]
Assume that the symplectic isoperimetric conjecture holds. Then, from inequality~$(\ref{ineq-vitconj})$ and Theorem~\ref{thm-about-capacities} it follows that
$$ {\frac {4^n} {\pi^n}} = {\frac {c^n_{_{\rm HZ}}(K \times K^{\circ})} {\pi^n}}
 \leq   {\frac {{\rm Vol}(K \times K^{\circ}) } {{\rm Vol}(B^{2n})} }  =  {\frac {n! \, {\rm Vol}(K \times K^{\circ})} {{\pi^n}}
},
$$
which  is exactly the lower bound for  ${\rm Vol}(K \times
K^{\circ})$ required by Mahler's conjecture.
\end{proof}

\begin{remark} {\rm
 It is clear from the above argument that we do not require the full strength of the symplectic isoperimetric conjecture to deduce Mahler's conjecture. Namely,
it is enough to know that the symplectic isoperimetric conjecture holds for the Hofer--Zehnder capacity, and for the special class of
convex domains in ${\mathbb R}^{2n}$ of the form $K \times K^{\circ}$, where $K \in {\cal K}_s({\mathbb R}^n_q)$.

}
\end{remark}
\noindent{\bf Structure of the paper:}
In Section~\ref{sec-HZ-Mink-billiards} we explain the relation between the Hofer--Zehnder capacity and the minimum length 
of periodic billiard orbits. In Section~\ref{sec:geometric facts} we provide the main geometric ingredients of the proof
of Theorem~\ref{thm-about-capacities} and Theorem~\ref{thm-about-strictly-convex-billiards}, which are given in Section~\ref{sec-proof-of-main-result}.

\noindent{\bf Acknowledgments:}
We thank the referees for their useful comments. The first named author was partially supported by ISF grant No. 247/11. The second named author was supported by the Dynasty Foundation, the President's of Russian Federation grant MD-352.2012.1, and the Russian government project 11.G34.31.0053. The third named author was partially supported by a Reintegration Grant SSGHD-268274 within the 7th European community framework programme, and by the ISF grant No. 1057/10.

\section{The Hofer-Zehnder Capacity and Minkowski Billiards } \label{sec-HZ-Mink-billiards}

In this section we describe the relation established in~\cite{AAO1} between the Hofer--Zehnder capacity~\cite{HZ}, restricted to the class of convex domains, and the minimal length of periodic 
Minkowski billiard trajectories.  For the reader's convenience,  we recall first some of the relevant definitions and notations. For a detailed exposition and proofs, see~\cite{AAO1}.

The restriction of the symplectic form
$\omega_{\rm st}$ to a smooth closed connected hypersurface $\Sigma
\subset {\mathbb R}^{2n}$ defines a 1-dimensional subbundle
${\rm ker}(\omega_{\rm st}|\Sigma)$ whose integral curves comprise the
characteristic foliation of $\Sigma$. In other words, a {\it closed
characteristic} $\gamma$ on $\partial \Sigma$ is an embedded circle
in  $\partial \Sigma$ tangent to the characteristic line bundle
\begin{equation*} {\mathfrak S}_{\Sigma} = \{(x,\xi) \in T
\partial \Sigma \ | \ \omega_{\rm st}(\xi,\eta) = 0 \ {\rm for \ all} \ \eta \in T_x
\partial \Sigma \}. \end{equation*}
The classical geometric problem of finding a closed characteristic 
has a well-known dynamical interpretation: if  the boundary
$\partial \Sigma$ is represented as a regular energy surface $\{x
\in {\mathbb  R}^{2n} \ | \ H(x) = {\rm const} \}$ of a smooth
Hamiltonian function $H : {\mathbb  R}^{2n} \rightarrow {\mathbb
R}$, then the restriction to $\partial \Sigma$ of the Hamiltonian
vector field $X_H$, defined by $i_{X_H} \omega_{\rm st} = -dH$, is a section
of ${\mathfrak S}_{\Sigma}$. Thus, the image of the periodic
solutions of the classical Hamiltonian equation $\dot x= X_H(x) = J\nabla H(x)$ on
$\partial \Sigma$ are precisely the closed characteristics of
$\partial \Sigma$. 
Recall that the action 
of a closed curve $\gamma$ 
is defined by $A(\gamma) = \int_{\gamma} \lambda$,
where  $\lambda =pdq$ is the Liouville 1-form, whose differential is
$d \lambda = \omega_{\rm st}$. Also, the action spectrum of $\Sigma$ is defined as
\begin{equation*}  {\cal L}(\Sigma) = \left \{ \, | \, {A}({\gamma}) \,  | \, ;
\, \gamma \ {\rm closed \ characteristic \ on} \ \partial \Sigma
\right \}.\end{equation*}

The following theorem, which serves here also as the definition 
of the Hofer--Zehnder capacity for the class of smooth convex bodies, 
can be found, e.g., in~\cite{HZ}.
\begin{theorem} \label{Cap_on_covex_sets} Let $\Sigma\subseteq {\mathbb
R}^{2n}$ be a convex bounded domain with smooth boundary $\partial
\Sigma$. Then there exists at least one closed characteristic $\gamma^*
\subset \partial \Sigma$ satisfying
\begin{equation*} 
 c_{_{\rm HZ}}(\Sigma)= { A}(\gamma^*) =  \min {\cal
L}( \Sigma). \end{equation*}
\end{theorem}

We remark that although the above definition of closed characteristics, as well as Theorem~\ref{Cap_on_covex_sets}, 
were given only for the class of convex bodies with smooth boundary, they can naturally be generalised 
to the class of convex sets in ${\mathbb R}^{2n}$ with non-empty interior  (see~\cite{AAO1}).

We now switch gears and turn to mathematical billiards in Minkowski geometry.
The general study of billiard dynamics in Finsler  and Minkowski
geometries was initiated in~\cite{GT}. From the point of view of geometric optics, Minkowski billiard
trajectories describe the propagation of waves in a homogeneous,
anisotropic medium that contains perfectly reflecting mirrors (see~\cite{GT}).
Below, we focus on the special case of Minkowski billiards in a smooth convex
body  $K \subset {\mathbb R}^n$. Roughly speaking, we equip $K$
with a metric given by a certain norm $\| \cdot \|$, and
consider billiards in $K$ with respect to the geometry induced by $\|
\cdot \|$.

More precisely, let $K \subset {\mathbb R}^n_q$, and $T \subset {\mathbb R}^n_p$ be two convex bodies with smooth boundary, and consider the 
unit cotangent bundle
\begin{equation*}
U_T^*K := K \times T = \{ (q,p) \, | \, q \in K, \ {\rm and} \
g_T(p)  \leq 1 \} \subset T^* {\mathbb R}^n_q  = {\mathbb R}^{n}_q
\times  {\mathbb R}^{n}_p. \end{equation*} Here $g_T$ is the gauge function of $T$ i.e., $g_T(x) = \inf \{r \, | \, x \in rT \}$, and in particular when $T$ is centrally symmetric i.e., $T=-T$, one has $g_T(x) = \|x\|_T$.
For $p \in \partial T$, the gradient vector $\nabla g_T(p)$ is the outer normal to $\partial T$ at the point $p$,  and is naturally considered to be in $\mathbb R^n_q$.

Motivated by the classical correspondence between closed
geodesics in a Riemannian manifold  and closed
characteristics of its unit cotangent bundle, the following definition of 
$(K,T)$-billiard trajectories, which are essentially closed billiard trajectories in $K$ when the bouncing rule
is determined by the geometry induced from the body $T$, was given in~\cite{AAO1}.

\begin{definition} \label{def-of-periodic-traj}
A closed $(K,T)$-billiard trajectory is the image of a piecewise smooth
map $\gamma \colon S^1 \rightarrow \partial (K \times T) $
such that for every  $t \notin {\mathcal B}_{\gamma}:= \{ t
\in S^1 \, | \, \gamma(t) \in \partial K \times \partial T \}$ one has
\begin{equation*}
\dot \gamma(t) = d \, {\mathfrak X}(\gamma(t)),  \end{equation*} for some positive
 constant $d$ and the vector field ${\mathfrak X}$  given by
\begin{equation*}
{\mathfrak X}(q,p) = \left\{
\begin{array}{ll}
(-\nabla g_T(p) ,0), &  (q,p) \in int(K) \times \partial T,\\
(0,\nabla g_K(q)), & (q,p) \in \partial K \times int(T).
\end{array} \right.
\end{equation*}
Moreover, for any $t \in {\mathcal B}_{\gamma}$, the left and right
derivatives of $\gamma(t)$ exists, and
\begin{equation} \label{eq-the-cone}
\dot \gamma^{\pm}(t) \in \{   \alpha (-\nabla g_T(p) ,0) + \beta
(0,\nabla g_K(q))    \ | \ \alpha,\beta \geq 0,  \ (\alpha, \beta) \neq (0,0) \}.
\end{equation}
\end{definition}

\begin{remark}
{\rm Although  in Definition~\ref{def-of-periodic-traj} there is a natural symmetry between the  bodies $K$ and $T$, 
in what follows  $K$ shall play the role of the billiard table, while  $T$ induces the geometry that governs the billiard dynamics in $K$.
It will be useful to introduce the following notation: 
for a $(K,T)$-billiard trajectory $\gamma$, the curve $\pi_q(\gamma)$,  where $\pi_q \colon {\mathbb R}^{2n} \rightarrow {\mathbb R}^n_q$ is the natural projection, 
shall be called a $T$-billiard trajectory in $K$.
} \end{remark}

\begin{definition} 
A closed $(K,T)$-billiard trajectory $\gamma$ is said to be {\it proper}
if the set ${\mathcal B}_{\gamma}$ is finite, i.e., $\gamma$ is a
broken characteristic that enters, and instantly exits, the product
$\partial K \times \partial T$ at the reflection points.
In the case where ${\mathcal B}_{\gamma} = S^1$, i.e., $\gamma$ is traveling
solely along the product $\partial K \times \partial T$,
 we say that $\gamma$ is a gliding trajectory.
\end{definition}

\begin{figure} 
\begin{center}
\begin{tikzpicture}[scale=1]

 \draw[important line][rotate=30] (0,0) ellipse (75pt and 40pt);

 \path coordinate (w1) at (2.1,4*0.75) coordinate (q0) at
 (-4.5*0.5,-2*0.2) coordinate (q1) at (1.6,4*0.44) coordinate (q2) at
 (1.74,+0.46) coordinate (w2) at (2.6,-1.1) coordinate (w3) at (-3.2,-0.11) coordinate (K) at
 (0,0.15);

\draw[red] [important line] (q0) -- (q1); \draw[->] (q1) -- (w1);
\draw[red] [important line] (q1) -- (q2); 
\draw[->] (q0) -- (w3);

\filldraw [black]
  (w3) circle (0pt) node[above] {{\footnotesize $w_2=\nabla \|q_2\|_K$}}
    (w1) circle (0pt) node[right] {{\footnotesize $w_1=\nabla \|q_1\|_K$}}
    (w2) circle (0pt) 
     (q0) circle (2pt) node[below right] {{\footnotesize $q_2$}}
      (q1) circle (2pt) node[above right=0.5pt] {{\footnotesize $q_1$}}
       (q2) circle (2pt) node[right=0.5pt] {{\footnotesize $q_0$}}
        (K) circle (0pt) node[right=0.5pt] {$K$};

       \begin{scope}[xshift=7cm]

 \draw[important line][rounded corners=10pt][rotate=10] (1.8,0) --
 (0.8,1.8)-- (-0.8,1.8)--  (-1.8,0)--  (-0.8,-1.8) -- (0.8,-1.8) --
 cycle;

 \path coordinate (p1) at (0.5,-4*0.433) coordinate (np0) at
 (2.3,2*0.9) coordinate (p0) at (1.2,2*0.53) coordinate (np1) at
 (0.7,-3) coordinate (p2) at (-1.2,2*0.66) coordinate (D) at
 (0.3,0.22);

 \draw[<-] (np0) node[right] {{\footnotesize $v_1=\nabla \|p_1
 \|_T$}} -- (p0);
 \draw[<-] (np1) node[right] {{\footnotesize $v_0=\nabla \|p_0
 \|_T$}} -- (p1);

 \draw[blue][important line] (p0) -- (p1);
 \draw[blue][important  line] (p0) -- (p2);

  \filldraw [black]
       (p1) circle (2pt) node[below right] {{\footnotesize $p_0$}}
         (p2) circle (2pt) node[left] {{\footnotesize $p_2$}}
         (p0) circle (2pt) node[above=2pt] {{\footnotesize $p_1$}}
          (D) circle (0pt) node[left] {$T$};
 \end{scope}
 \end{tikzpicture}

 \caption{A proper $(K,T)$-Billiard trajectory.} 
 \end{center}
 \end{figure}
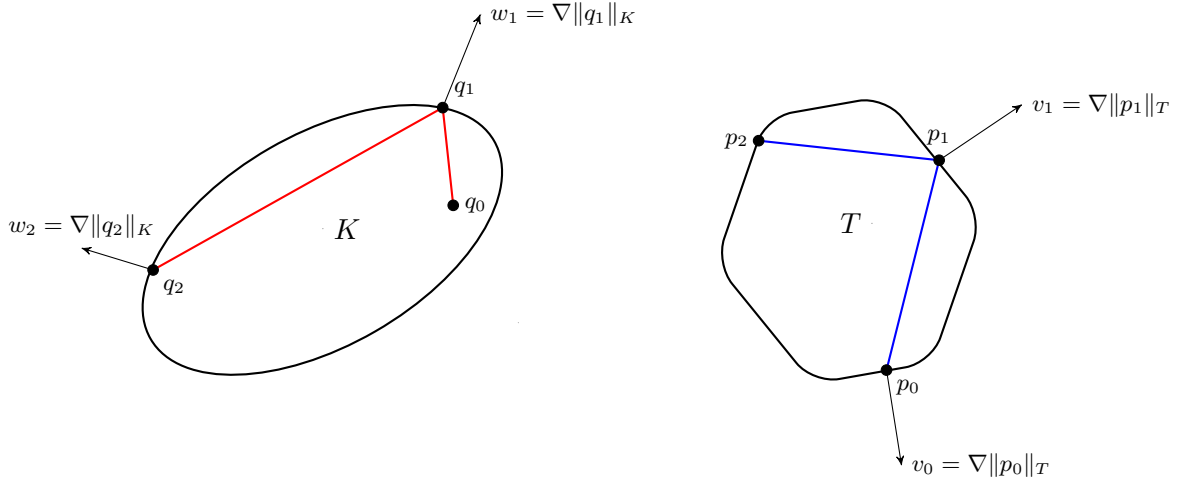

For a proper billiard trajectory, when we follow the flow of the vector field
${\mathfrak X}$, we move in $K \times \partial T$ from $(q_0,p_0)$ to
$(q_1,p_0)$ following the opposite of the outer normal to
$\partial T$ at $p_0$. When we hit the boundary $\partial K$ at the
point $q_1$, the vector field  changes, and we start to move in
$\partial K \times T$ from $(q_1,p_0)$ to $(q_1,p_1)$ following the outer
 normal to $\partial K$ at the point $q_1$.  Next, we move from
$(q_1,p_1)$ to $(q_2,p_1)$ following the opposite of the normal to
$\partial T$ at $p_1$, and so forth  (see
Figure $1$). Note that this reflection law is a
natural variation of the classical one (i.e., equal impact and
reflection angles) when the Euclidean structure on ${\mathbb R}^n_q$
is replaced by the metric induced by the norm $\| \cdot \|_{T}$.
Moreover, it is not hard to check that when $T$ is the Euclidean
unit ball, the bouncing rule described above is the
classical one. Also, similarly to the Euclidean case, one can check that  $(K,T)$-billiards 
correspond to critical points of   the length functional given by  the support function $h_T$, where $h_T(u) = \sup \{ \langle x,u \rangle \, ; \, x \in T \}$. 
%

We remark that in~\cite{AAO1} is was proved that 
every $(K,T)$-billiard trajectory is either a proper trajectory or a gliding one, and that the following holds:

\begin{theorem}[\cite{AAO1}] \label{Main-Theorem-From-AAO1}
Let $K \in {\cal K}({\mathbb R}_q^n)$ and $T \in {\cal K}({\mathbb R}_p^n)$ be
two smooth strictly convex bodies. Then the Hofer--Zehnder capacity $c_{_{\rm HZ}}(K \times T)$ equals the length, with respect to the support function $h_T$, of the shortest periodic $T$-billiard trajectory in $K$. 
\end{theorem}

\section{Two Geometric Facts}\label{sec:geometric facts}        

Here we provide the geometric ingredients needed for the proof of Theorem~\ref{thm-about-capacities},
which may be of independent interest in the realm of convex geometry.
Throughout this section, we assume that $K \subset {\mathbb R}^n$ is a centrally symmetric convex body. Moreover, 
it will be convenient to use the following notation. For a set of points $X$, we denote by ${\rm Conv}(X)$ the convex hull of $X$; and
for a closed oriented polygonal path ${\cal P} \subset {\mathbb R}^{n}$, specified by vertices $x_1, \ldots, x_{m}\in \RR^n$,  we set 
\[{\rm {\rm Length}_K({\cal P}):=
Length}_K(x_1 x_2 \cdots x_{m}):=  \|x_1 - x_{m} \|_K + \sum_{i=1}^{m-1} \| x_{i+1}-x_i
\|_K.\]

\subsection{First Geometric Fact}\label{sec:0inconv}

\begin{theorem}\label{thm-about-Klengths}
Let $K\subset \RR^n$ be a centrally symmetric convex body, and consider points
$\{x_1, \ldots, x_{m}\}\subset \RR^n\setminus {\rm int}(K)$ with $m\ge 2$, $x_i\neq x_{i+1}$ for $i=1,\ldots, m-1$ and $x_1\neq x_m$, such that $0 \in {\rm Conv}\{x_i\}_{i=1}^{m}$. 
Then, $$
{\rm Length}_K(x_1 x_2 \cdots x_{m}) \geq 4.$$
Moreover, this minimium is attained for $m=2$ and $x_1=-x_2\in \partial K$, and when the body $K$ is strictly convex, this is the only equality case. 
\end{theorem}

In order to prove Theorem~\ref{thm-about-Klengths} we
shall need the following  lemmas.  

\begin{lemma}\label{lem:length-of-path-between-xpy-p}
Let $K$ be a centrally symmetric convex body, and let ${\cal P}$ be an
oriented polygonal closed path in ${\mathbb R}^{2n}$ which satisfies the following two
properties:

\begin{itemize}

\item It passes through two points, $x,y$, outside
$K$ i.e., $\|x\|_K\ge 1, \|y\|_K \ge 1$,

\item It passes through two points $z, -z$, such that $x$ and $y$ do not lie in 
the same connected component of ${\cal P} \setminus \{-z,z\}$.
\end{itemize}
Then,  ${\rm Length}_K({\cal P}) \ge 4$.
\end{lemma}

\begin{proof}[{\bf Proof of Lemma~\ref{lem:length-of-path-between-xpy-p}}] 
We will show that each of the two parts of the polygonal path, 
between the points $z$ and
$-z$, has length at least $2$. Indeed, call one of these paths ${\cal P}'$, then 
consider the new path ${\cal P}' \cup -{\cal P}'$, which is a closed path that
joins $x$ and $-x$. By the triangle inequality, the length
of each part between $x$ and $-x$ is at least $\|x - (-x)\|_K \ge 2$.
Since we have doubled the path ${\cal P}'$, we get that ${\rm Length}_K({\cal P}') \ge 2$.
In a similar way, the other part of the path ${\cal P}$ joining $-z$ and $z$ has length
at least $2$, and we conclude that ${\rm Length}_K({\cal P}) \ge 4$ as claimed, and the proof of the lemma is complete.  
\end{proof}

\begin{lemma} \label{Lemma-about-short-cuts} Let $x_1,\ldots,x_m$ be 
points in ${\mathbb R}^n$, 
and  let $z \in {\rm Conv}\{x_1,\ldots,x_m\}$. Then $${\rm
Length}_K(x_1x_2\cdots x_m) \geq {\rm Length}_K(x_1 z x_m).$$
\end{lemma}

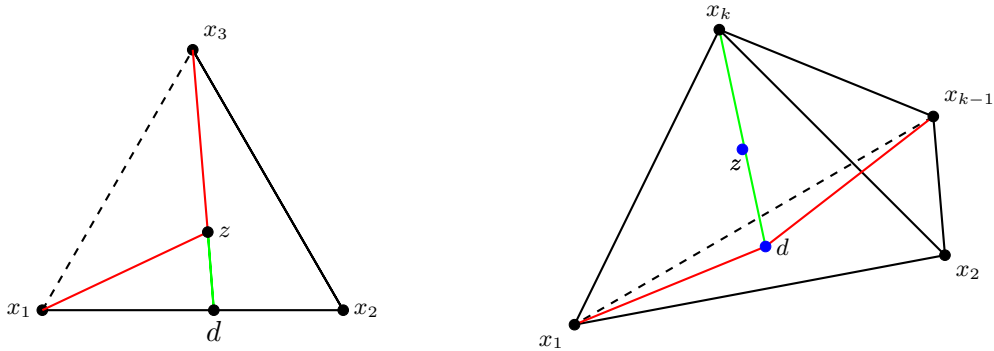
\begin{figure}[h!]
\begin{center}
  \begin{tikzpicture}[scale=1]

\path coordinate (c1) at (4*0.5,-4*0.433) coordinate (c2) at
(-4*0.5,-4*0.433) coordinate (c3) at (0,4*0.433); \draw[important
line] (c1) node[right] {{\footnotesize $x_2$}} --  (c3) node[above
right] {{\footnotesize $x_3$}} -- (c1) -- (c2);

\draw[dashed] [important line] (c2) node[left] {{\footnotesize
$x_1$}} -- (c3);

 \filldraw [black]
      (c1) circle (2pt)
       (c2) circle (2pt)
        (c3) circle (2pt);

\path coordinate (z) at (2*0.1,-2*0.3464) coordinate (d1) at
(4*0.07,-4*0.433) coordinate (d2) at (4*0.2,4*0.0866) coordinate
(d3) at (-4*0.1,4*0.259 ); \draw[important line, green] (d1)
node[below, black] {$d$} -- (z) node[above right, black]
{{\footnotesize }}-- (d1);

\draw[important line, red] (c2)  -- (z) node[right, black]
{{\footnotesize $z$}}-- (c3);

 \filldraw [black]
      (d1) circle (2pt)
       (d2) circle (0pt)
       (z) circle (2pt);

      \begin{scope}[xshift=7cm]

\path coordinate (d11) at (1,-0.5,1) coordinate (z11) at
(0.5,0.6,0.5) node[above right, black] {{\footnotesize $z$}}
coordinate (a11) at (4,2,3) node[above right, black] {{\footnotesize
$z$}} coordinate (a22) at (3,-1,0) coordinate (a33) at (0,2,0)
 coordinate (a44) at
(0,0,5); \draw [important line] (a11) node[above right , black]
{{\footnotesize $x_{k-1}$}}  -- (a22) ; \draw[important line] (a11)
-- (a33); \draw[dashed] [important line] (a11) -- (a44);
\draw[important line] (a22) -- (a33);

\draw[important line] (a22) node[below right , black]
{{\footnotesize $x_2$}} -- (a44) node[below left , black]
{{\footnotesize $x_1$}};

 \draw[important line] (a33)node[above , black]
{{\footnotesize $x_k$}} --
 (a44);

 \draw[important line, green] (d11) node[right, black]
{{\footnotesize $d$}} -- (a33);

 \draw[important line, red] (a44) -- (d11);
  \draw[important line, red] (d11) -- (a11);

\filldraw [blue]
 (z11) circle (2pt)
  (d11) circle (2pt);
  \filldraw [black]
      (a11) circle (2pt)
       (a22) circle (2pt)
       (a44) circle (2pt)
       (a33) circle (2pt);

\end{scope}
\end{tikzpicture}

\caption{The $2$-dimensional case (left), and the $k$-dimensional
case (right).}
\end{center}
\end{figure}

\begin{proof}[{\bf Proof of Lemma~\ref{Lemma-about-short-cuts}}] We proceed by induction on the number $m$. The case $m=2$ is trivial.
Next, we assume that the claim holds
for $m-1$. Since $z\in {\rm Conv}\{x_1,\ldots,x_m\}$, 
it can be written as a convex combination of the form 
\[ z = \sum_{j=1}^m \lambda_j x_j, \  {\rm where}  \ \ \sum_j^m \lambda_j =1 \ {\rm and} \  \lambda_j \ge 0.\]
We put $d=\sum_{j=1}^{m-1} \frac{\lambda_j}{1-\lambda_m} x_j $. 
Thus, $z=\lambda_m x_m + (1-\lambda_m)d$ (see Figure 2). In particular, one has
\begin{equation}  \label{eq1-in-lemma-3.3}  \|x_m-d\|_K = \|x_m -z\|_K + \|d-z\|_K. \end{equation}
By the inductive hypothesis, 
${\rm Length}_K(x_1dx_{m-1}) \leq {\rm Length}_K(x_1 x_2 \cdots x_{m-1})$. 
Equivalently, 
\begin{equation} \label{eq2-in-lemma-3.3} \|d-x_1\|_K + \|x_{m-1}-d\|_K \le \sum_{j=1}^{m-2}\|x_{j+1}-x_j\|_K. \end{equation}
Combining~$(\ref{eq2-in-lemma-3.3})$ and~$(\ref{eq2-in-lemma-3.3})$, we conclude that
\begin{eqnarray*} {\rm Length}_K(x_1zx_m)  &=& \|z-x_1\|_K+ \|x_m -z\| _K +\|x_1-x_m\|_K \\
&\le& \|d-x_1\|_K + \|z-d\|_K+ \|x_m -z\| _K +\|x_1-x_m\|_K \\
&=& \|d-x_1\|_K + \|x_m-d\|_K+ \|x_1-x_m\|_K \\
&\le &  \|d-x_1\|_K + \|x_{m-1}-d\|_K+ \|x_m-x_{m-1}\|_K + \|x_1-x_m\|_K \\
& \le & \sum_{j=1}^{m-2}\|x_{j+1}-x_j\|_K +\|x_m-x_{m-1}\|_K + \|x_1-x_m\|_K\\
& = &{\rm Length}_K(x_1x_2\cdots x_{m}).
\end{eqnarray*}
The proof of Lemma~\ref{Lemma-about-short-cuts} is thus complete. \end{proof}

\noindent {\bf Remark:} In the next lemma we denote the indices $j =
j({\rm mod}(m))$, so that $m+1 = 1$. 
Also, by $\{x_k\}_{k=j}^i$, where $i<j$, we mean $\{x_1,\ldots, x_i\} \cup \{x_j,\ldots, x_m\}$. 


\begin{lemma}\label{lem:conv-intersect-minusconv}
Let $x_1, \ldots, x_{m}\in \R^n$, where $m\ge 3$, such that $0 \in
{\rm Conv}(\{x_i\}_{i=1}^{m})$. Then there exist two indices $i_0<j_0$
for which \begin{equation}  \label{nonemptyintersection} \conv
(\{x_k\}_{k=i_0}^{j_0}) \bigcap -\conv(\{x_k\}_{k=j_0}^{i_0-1}) \neq \emptyset.
\end{equation}
\end{lemma}

\begin{proof}[{\bf Proof of Lemma~\ref{lem:conv-intersect-minusconv}}] Since $0 \in
{\rm Conv}(\{x_i\}_{i=1}^{m})$,
there are $\eta_i \ge 0$ such that
\begin{equation} \label{uniquepres} \sum_{i=1}^{m} \eta_i x_i =
0 \ \  {\rm and} \ \ \sum_{i=1}^{m} \eta_i = 1.
\end{equation} 
We assume, without loss of generality, that $\eta_1,\eta_m < 1/2$, and we shall show that the claim holds with 
$i_0=1$. 
We choose $j$ to be the maximal
possible index so that  $ \sum_{i=1}^{j-1}\eta_{i}\le 1/2$. 
In particular, $ \sum_{i=1}^{j}\eta_{i}> 1/2$.
Further, since $\eta_m <1/2$, we know that $j\neq m$. 
Let $\eta_j = \eta_j'+\eta_j''$ such that $ \sum_{i=1}^{j-1}\eta_{i} + \eta_j' = 1/2$. 
Of course in this case  \[ \sum_{i=j+1}^{m}\eta_{i} + \eta_j'' = 1/2.\] 
Since $\sum_{i=1}^{m} \eta_i x_i =
0$, one has  
\[  \sum_{i=1}^{j-1} \eta_i x_i + \eta_j'x_j = -\Bigl ( \sum_{i=j+1}^{m} \eta_i x_i + \eta_j''x_j \Bigr).  \] 
On both sides the coefficients  sum to $1/2$, so that letting 
\[ p = 2 \Bigl (  \eta_j'x_j  + \sum_{i=1}^{j-1} \eta_i x_i  \Bigr ), \]
we get $p\in \conv (\{x_i\}_{i=1}^{j})$ and $-p \in \conv(\{x_i\}_{i=j}^m)$, and the claim is proved.
\end{proof} 

\begin{remark}
{\rm Note that in fact we proved something slightly stronger, namely, that in each of the subsets $\{x_{i_0}, \ldots ,x_{j_0}\}$, $\{x_{j_0},\ldots,x_{i_0-1}\}$ there are at least 2 points. } \end{remark} 

\begin{proof}[{\bf Proof of Theorem \ref{thm-about-Klengths}.}]
Let $\{x_1,\ldots,x_{m}\}\subset \RR^n\setminus {\rm int}(K)$ with $m\ge 2$, and such that $0 \in {\rm Conv} (\{x_i\}_{i=1}^{m})$. 
We use Lemma \ref{lem:conv-intersect-minusconv} to find $1\le i_0< j_0\le m $ and $p\in \RR^n$ such that 
\[ p\in \conv ( \{ x_k \}_{k=i_0}^{j_0} ) \  \  {\rm and} \  -p \in \conv ( \{x_k\}_{k=j_0}^{i_0-1}). \] 
From Lemma \ref{Lemma-about-short-cuts}  it follows that
$$\sum_{k=1}^{j_0-i_0} \|x_{i_0+k}-x_{i_0+k-1}\|_K \geq \|p-x_{i_0} \|_K+\|x_{j_0}-p\|_K.$$
Note that $-p$ is in ${\rm Conv} (\{x_k \}_{k=j_0}^{i_0-1})$  and therefore evidently in  ${\rm Conv} ( \{x_k\}_{k=j_0}^{i_0})$. Hence, 
$$\sum_{k=1}^{i_0-j_0} \|x_{{j_0}+k}-x_{j_0+k-1}\|_K \geq \|-p-x_{j_0}\|_K+\|x_{i_0}-(-p)\|_K.$$
Combining the last two inequalities yields, 
\[ \sum_{k=1}^{m} \|x_{k}-x_{k-1}\|_K \geq \|p-x_{i_0} \|_K+\|x_{j_0}-p\|_K+\|-p-x_{j_0} \|_K+\|x_{i_0}-(-p)\|_K.\]
We thus bound from below the length we are considering by the length of the closed path between $(x_{i_0},p,x_{j_0},-p)$, where $\|x_{i_0} \|_K\ge 1$ and $\|x_{j_0}\|_K\ge 1$. 
By Lemma~\ref{lem:length-of-path-between-xpy-p}, the length of this latter path, with respect to $\| \cdot \|_K$,  is at least 4. 
 
 In the strictly convex case, where triangle inequalities are sharp unless the points are on the same line, 
going through the proof one may easily analyse the equality case, which is when there are only two points, on $\partial K$, which are antipodal.

\end{proof}


\subsection{Second Geometric Fact}\label{sec:noneedfor0}

\begin{theorem}\label{thm-about-Klengths-withoutconv}
Let $K\subset \RR^n$ be a smooth centrally symmetric convex body, and consider points
$\{q_1, \ldots, q_{m}\}\subset \partial K$ with $m\ge 2$, such that $0 \in \mathop{\rm Conv}\{n_K(q_i)\}_{i=1}^{m}$, where $n_K(q_i) = \nabla g_K(q_i)$ is the outer normal to $K$ at $q_i$.
Then, 
$$
{\rm Length}_K(q_1 q_2 \cdots q_{m}) \geq 4.
$$
\end{theorem}

The proof shall make use of Theorem \ref{thm-about-Klengths}. We shall construct a new polygonal trajectory, $(q_1' q_2' \cdots q_{m'}')$ at most as long as the original one, such that $0\in   {\rm Conv}\{q_i'\}_{i=1}^{m'}$.

We need the following simple lemma about simplices:
 
\begin{lemma} \label{lemma-about-simplices}
\label{lemma-simplexcover}
Let $S$ denote a non-degenerate simplex in ${\mathbb R}^k$ with vertices $\{x_i\}_{i=1}^{k+1}$, and let $\{q_i\}_{i=1}^{k+1}\subset {\mathbb R}^k$ such that $q_i \in {\rm Conv}(\{x_j\}_{j\neq i})$.  Assume $0\in {\rm Conv}(\{x_i\}_{i=1}^{k+1})$. Then, for some  subset $I\in \{1,\ldots,n+1\}$ of cardinality at least 1, the origin $0$ lies in the convex hull of the union $\{q_i\}_{i\in I}\cup \{x_j\}_{j\not\in I}$. 
\end{lemma}

\begin{remark} {\rm 
It is easy to check that the cardinality of $I$ is at least $2$, whenever the origin $0$ is in the interior of $S$. Indeed, 
for $I=\{i_0\}$ we have that ${\rm Conv}(q_{i_0} \cup \{x_j\}_{j\neq i_0}) = {\rm Conv}( \{x_j\}_{j\neq i_0}) $ is a facet of $S$. 
}
\end{remark}

\begin{proof}[{\bf Proof of Lemma~\ref{lemma-about-simplices}}]
Let $S_I$ be the possibly degenerate simplex given by the convex hull $\mathop{\rm Conv}\left(\{q_i\}_{i\in I}\cup \{x_j\}_{j\not\in I}\right)$. If we add the ``forbidden'' simplex $S_{\emptyset } (=S)$ to the collection $\{S_I: {I\subset \{1,\ldots,k+1\}, |I|\ge 1}\}$, then these simplices together can be viewed as a piecewise linear image of the boundary of a $(k+1)$-dimensional combinatorial cross-polytope $C$. 

More precisely, we consider an abstract cross-polytope $C$ spanned by $\pm e_1, \ldots, \pm e_{k+1}\in \RR^{k+1}$ and map its vertices to $\RR^k$ so that $f(-e_i) = x_i$ and $f(+e_i) = q_i$. Then we extend the map $f$ to the boundary $\partial C$ piecewise linearly. This boundary $\partial C$ is a piecewise linear sphere of dimension $k$ and the degree of this piecewise linear map $f : \partial C\to \RR^k$ is well defined and equal to zero. Hence, as any point in the interior of $S$ is covered by $S_{\emptyset}=f(\conv\{-e_1,\ldots, -e_{k+1}\})$ (which is equal to $S$), it must also be covered by some other $S_I=f(\conv\left(\{+e_i\}_{i\in I}\cup \{-e_j\}_{j\not\in I}\right))$ with $|I| \ge 1$.
\end{proof}

\begin{proof}[{\bf Proof of Theorem \ref{thm-about-Klengths-withoutconv}}]
We are given a set of points $\{q_i\}_{i=1}^m \in \partial K$ such that $0$ is in the convex hull of their normals $n_K(q_i) = \nabla g_K(q_i)$. 

First, we invoke the Carath\'eodory theorem and reduce to the subset $\{q_i\}_{i\in J} \in \partial K$ of size $|J|\le n+1$, such that $0$ is still in the convex hull of $\{n_K(q_i)\}_{i\in J}$. Moreover, we consider the inclusion minimal $J$ with this property, so that $0$ is in the relative interior of the convex hull of $\{n_K(q_i)\}_{i\in J}$. It is clear that the closed polygonal line through this subset  $\{q_i\}_{i\in J}$ in the same cyclic order is not longer than the original one, by the triangle inequality. So we consider this new relabeled set $\{q_i\}_{i=1}^{m'}$.

Next, we claim that there is no loss of generality in assuming that the vectors  $\{n_K(q_i)\}_{i=1}^{m'}$ positively span ${\mathbb R}^n$. Indeed, if not, then we may 
project onto the quotient by the annihilator of the subspace spanned by the normals $\{n_K(q_i)\}_{i=1}^{m'}$. Under this projection, the length of the closed polygonal line cannot increase, if we use the induced norm on the quotient. The projections of $q_i$ remain on the boundary, since the respective normals are still norming functionals. 
This all means that we may assume without loss of generality that $m'=n+1$. We thus conclude that, considering the half-spaces supporting $K$ at $q_i$, we may assume that their intersection is an $n$-dimensional simplex. We denote its vertices by $\{x_i\}_{i=1}^{n+1}$, in such a way that $q_i \in {\rm Conv}(\{x_j\}_{j\neq i})$.

We claim that the length ${\rm Length}_K(q_1q_2\cdots q_{n+1})$ is at least 4. To this end we shall replace this polygonal path $q_1q_2\cdots q_{n+1}$ by a shorter one as follows. By Lemma \ref{lemma-simplexcover}, there is a subset $I$ of size at least 2 such that $0$ is in the convex hull of $\{q_i\}_{i\in I}$ and $\{x_j\}_{j\not\in I}$. Note that if no $x_j$'s participate, we are done by Theorem \ref{thm-about-Klengths}. We may thus assume that $|I|<n+1$.

In particular, there exists a convex combination $x$ of $\{x_j\}_{j\not\in I}$ and a convex combination $q$ of $\{q_i\}_{i\in I}$ such that $0 = (1-\lambda)q+\lambda x$ for some $\lambda \in (0,1)$.
 
We shall first consider the shorter trajectory, which passes only through $\{q_i\}_{i\in I}$ (in the same order as before). By the triangle inequality, it is not longer than the original trajectory. Then, consider the new trajectory passing through the new points $q_i' = (1-\lambda)q_i + \lambda x$. Note that $0$ is in the convex hull of $q_i'$. Further, note that $q_i'$ is a convex combination of two points which are on the facet of $S$ that is opposite to $x_i$, since $x$ is composed of $x_j$ with $j\not\in I$ and in particular $j\neq i$. Therefore, by convexity, we see that $q_i'$ belongs to the hyperplane supporting $K$ via $q_i$ and in particular $q_i' \not\in \mathop{\rm int}(K)$. The length of the new path is of course $(1-\lambda)$ times the length of the path via $q_i$, and at the same time, by Theorem \ref{thm-about-Klengths}, this length is at least 4, and the proof of Theorem~\ref{thm-about-Klengths-withoutconv} is now complete.  
\end{proof}

\begin{remark} {\rm 
It is not difficult to check that the equality case in Theorem~\ref{thm-about-Klengths-withoutconv} is exactly  the same as in Theorem \ref{thm-about-Klengths} above. }
\end{remark}

\section{Proof of the Main Results} \label{sec-proof-of-main-result}

It is well known that the Hofer--Zehnder capacity $c_{_{\rm HZ}}$ 
is continuous with respect to the Hausdorff metric on the class of convex domains 
(see e.g.,~\cite{MS}, Exercise 12.7). Hence, we can assume without loss of generality that $K \in {\cal K}_s({\mathbb R}_q^{n})$ is  smooth, and strictly convex.

We shall make use of the following simple lemma regarding $(K,T)$-billiard trajectories. 

\begin{lemma}\label{lem:0inconvnormals}
Let $\gamma:S^1\to \partial (K\times T)$ be a $(K,T)$-billiard trajectory and let $\gamma_q = \pi_q(\gamma(S^1))$ denote its projection to ${\mathbb R}_q^n$. Then 
$$
0\in \mathop{\rm Conv}\bigl( n_K\left(\gamma_q \cap \partial K\right) \bigr), 
$$
In other words, the origin lies in the convex hull of the set of outer normals at those points where the trajectory meets the boundary of $K$. 
\end{lemma}

\begin{proof}[{\bf Proof of Lemma~\ref{lem:0inconvnormals}}]
For a proper billiard trajectory, the reflection rule at $q_i$ connects the normal and the momenta before and after $q_i$ via:
$$
\lambda_i n_K(q_i) = p_{i-1} - p_i,
$$
for some nonnegative $\lambda_i$. This follows from the definition of $(K,T)$-billiard trajectories, and also can be easily seen from the variational characterization of the trajectory. Indeed, apply the Lagrange multipliers to optimizing $h_T(q_i-q_{i-1}) + h_T(q_{i+1}-q_{i}) $ 
under the constraint $g_K(q_i)=1$. The multiplier $\lambda_i$ is nonnegative, since otherwise the trajectory would increase its length while moving $q_i$ in the direction of an interior point of the segment $[q_{i-1}, q_{i+1}]$, which is impossible by the triangle inequality. Summing up, and using the fact that the trajectory is closed, we obtain the required combination
$$
\sum_{i=1}^{m} \lambda_i n_K(q_i) = 0.
$$
For a gliding trajectory, it is known (see Proposition 2.2 in~\cite{AAO1}) that there exist an equation connecting the momentum and the outer normal of the form:
$$
\frac{d}{dt} p(t) = \lambda(t) n_K(q(t)),
$$
for some smooth positive function $\lambda : S^1 \rightarrow {\mathbb R}$. As in the case of a proper trajectory, this gives the required convex combination after integration.
\end{proof}

\begin{proof}[{\bf Proof of Theorem \ref{thm-about-capacities}}]

We shall use Theorem~\ref{thm-about-Klengths-withoutconv}, in order to show that any periodic $K^{\circ}$-billiard trajectory in $K$ has length (with respect to $h_{K^{\circ}} = \| \cdot \|_K$) at least 4, and that this bound is actually attained. Combining this with Theorem~\ref{Main-Theorem-From-AAO1} will prove Theorem~\ref{thm-about-capacities}.

The fact that $c_{_{\rm HZ}}(K\times K^{\circ})$ is at most $4$ follows from the easily verified fact that for any $q\in \partial K$, the path $[-q,q]$ is a $K^{\circ}$-billiard trajectory in $K$ (called a ``bouncing orbit''). This in turn  follows from the fact that for a strictly convex body $K$ one has  
\begin{equation} \label{formula-for-normal}
p = \frac{n_K(q)}{\|n_K(q)\|_{K^{\circ}}} \in \partial K^{\circ} \ \  {\rm if \ and \ only \ if \ \ } q = \frac{n_{K^{\circ}}(p)}{\|n_{K^{\circ}}(p)\|_K} \in \partial K.
\end{equation}
The lengths 
of these bouncing orbits are exactly $4$, and so $c_{_{\rm HZ}}(K\times K^{\circ}) \le 4$. 

To show that any $K^{\circ}$-billiard trajectory in $K$ has length at least $4$ (again, with respect to $\| \cdot \|_K$), we consider such a trajectory $\gamma_q=\pi_q(\gamma(S^1))$ and use Lemma \ref{lem:0inconvnormals} together with Carath\'eodory's theorem to find $n+1$ points $q_j\in \gamma_q \cap \partial K$, which satisfy $0\in {\rm Conv}(\{n_K(q_i)\}_{i=1}^{n+1})$. By the triangle inequality, the length of $\gamma_q$ 
is at least the length of the closed polygonal line $(q_1\cdots q_{n+1})$. By Theorem \ref{thm-about-Klengths-withoutconv}, this length is at least $4$, and the proof of Theorem \ref{thm-about-capacities} is complete. 
\end{proof}

\begin{proof}[{\bf Proof of Theorem \ref{thm-about-strictly-convex-billiards}}]
Note that in the proof of Theorem~\ref{thm-about-capacities} above we already showed that the $\| \cdot \|_K$-length of any periodic $K^{\circ}$-billiard trajectory in $K$ is at least 4, and this bound is clearly attained on $2$-bouncing orbits.  Moreover, as remarked after its proof, the equality conditions in Theorem~\ref{thm-about-Klengths-withoutconv} in the strictly convex
case are precisely those of Theorem~\ref{thm-about-Klengths}. Namely, equality is attained if and only if the trajectory is composed of two antipodal points in $\partial K$, which
is precisely $2$-bouncing orbits.
\end{proof}

\begin{remark}  \label{rem-about-K-T} {\rm
It follows from Theorem 1.7 that for any centrally symmetric convex bodies $K\subset \RR^n_q,$ and $T\subset \RR^n_p$, one has 
\begin{equation} \label{eq-for-cap-of-K-T} c_{_{\rm HZ}}(K \times T)   = \overline c(K \times T) =  4 \inrad_{T^{\circ}}(K), \end{equation}
where $\inrad_A(B) = \max\{ r \, | \, rA \subset B\}$. As before, since the above expressions are continuous with respect to the Hausdorff metric, one may assume that both $K$ and $T$ are smooth. 
To establish~$(\ref{eq-for-cap-of-K-T})$, we first use 
Theorem~\ref{thm-about-capacities}, the monotonicity property of symplectic capacities, and the fact that  $\overline c$ is the largest possible symplectic capacity: 
 if $rT^{\circ} \subset K$ then $T\supset rK^{\circ}$, and hence
\[  \overline c(K \times T) \geq   c_{_{\rm HZ}}(K\times T) \ge c_{_{\rm HZ}}(K \times rK^{\circ}) = 4r.\]
Next, we use the conformality  property of symplectic capacities to assume without loss of generality that $\inrad_{T^{\circ}}(K) = 1$. 
Let $q_0 \in \partial K \subset {\mathbb R}^n_q$ be a tangency point of $K$ and $T^{\circ}$, and set $p_0 \in \partial K^{\circ} \subset {\mathbb R}^n_p$  to be the point defined as in~$(\ref{formula-for-normal})$ above.
Note that $p_0$ is also a tangency point of $K^{\circ}$ and $T$. 
Moreover, as in the proof of Theorem 1.8, the $2$-bouncing trajectory $[-q_0,q_0]$ is not only a $K^{\circ}$-billiard trajectory in $K$, but also a $T$-billiard trajectory in $K$, as $K$ and $T^{\circ}$ share the same normals at the points $\pm q_0$, and 
 $K^{\circ}$ and $T$  share the same normals at  $\pm p_0$.  Now, by definition, the cylinder 
$$
\widetilde Z = \{(p, q) \in\mathbb R^{2n}: |q_0(p)| \le 1\ \text{and}\ |p_0(q)| \le 1\},
$$
contains the body $K \times T$. Moreover,  its Hofer-Zehnder capacity satisfies $c_{_{\rm HZ}}(\widetilde Z) \leq 4$ since the above  $2$-bouncing trajectory is a  closed characteristic on the boundary $\partial \widetilde Z$, of action $4$.
Finally, since $q_0(p_0) = p_0(q_0)=1$, the interior of the cylinder $\widetilde Z$ is symplectomorphic to the interior of the standard symplectic cylinder $B^2\left(\sqrt{\frac  4 {\pi}}\right) \times {\mathbb C}^{n-1}$, over which all symplectic capacities coincide.
Thus we conclude that $\overline c(K \times T) \leq \overline c(\widetilde Z) = c_{_{\rm HZ}}(\widetilde Z) \leq 4$. Thus we have established~$(\ref{eq-for-cap-of-K-T})$.
} 
\end{remark}

\bigskip

\noindent Shiri Artstein-Avidan,\\
\noindent School of Mathematical Science, Tel Aviv University, Tel Aviv, Israel.\\
\noindent {\it e-mail}: shiri@post.tau.ac.il

\bigskip

\noindent Roman Karasev, Dept. of Mathematics, Moscow Institute of Physics and Technology, Institutskiy per. 9, Dolgoprudny, Russia 141700\\
\noindent Roman Karasev, Institute for Information Transmission Problems RAS, Bolshoy Karetny per. 19, Moscow, Russia 127994\\
\noindent Roman Karasev, Laboratory of Discrete and Computational Geometry, Yaroslavl' State University, Sovetskaya st. 14, Yaroslavl', Russia 150000\\
\noindent {\it e-mail}: r\textunderscore n\textunderscore karasev@mail.ru

\bigskip

\noindent Yaron Ostrover,\\
\noindent School of Mathematical Science, Tel Aviv University, Tel Aviv, Israel.\\
\noindent {\it e-mail}: ostrover@post.tau.ac.il


\begin{thebibliography}{}
\bibitem{AAMO} Artstein-Avidan, S., Milman, V., Ostrover, Y. {\it The M-ellipsoid,
symplectic capacities and volume}, Comment. Math.
Helv. 83, (2008) no.2, 359--369.


\bibitem{AAO1} Artstein-Avidan, S., Ostrover Y. {\it  Bounds for Minkowski billiard trajectories in convex bodies,} 
Intern. Math. Res. Not. (IMRN) (2012) doi:10.1093/imrn/rns216.


\bibitem{Bl} Blaschke, W. {\it \"Uber affine Geometrie VII: Neue Extremeigenschaften
von Ellipse und Ellipsoid}, Ber. Verh. S\"achs. Akad. Wiss. Leipzig,
Math.-Phys. Kl {\bf 69} (1917) 306--318, Ges. Werke {\bf 3} 246--258.

\bibitem{BM} Bourgain, J., Milman, V. D. {\it New volume ratio properties for convex
symmetric bodies in ${\mathbb R}^n$,} Invent. Math. 88 (1987), no.
2, 319--340.

\bibitem{CHLS} Cieliebak, K., Hofer, H., Latschev, J., Schlenk
F. {\it Quantitative symplectic geometry,} in: Dynamics, ergodic theory,
and geometry, 1--44, Math. Sci. Res. Inst. Publ., 54, Cambridge Univ.
Press, Cambridge, 2007.

\bibitem{EH} Ekeland, I., Hofer, H. {\it Symplectic topology
and Hamiltonian dynamics I $\&$ II}, Math. Z. 200 (1989), no. 3, 355--378, and Math. Z.  203 (1990), no.4, 553--567.

\bibitem{FGS} Frauenfelder, U., Ginzburg, V., Schlenk, F. {\it Energy capacity
inequalities via an action selector},  Geometry, spectral theory,
groups, and dynamics, 129--152, Contemp. Math., 387, Amer. Math.
Soc., Providence, RI, 2005.


\bibitem{GPV} Giannopoulos, A., Paouris, G., Vritsiou, B. {\it The isotropic position and the reverse Santal\'o inequality,}
to appear in Israel J. Math.

\bibitem{GMR} Gordon, Y., Meyer,  M.,  Reisner, S. {\it Zonoids with minimal
volume product -- a new proof,} Proc. Amer. Math. Soc.
104 (1988), no. 1, 273--276.

\bibitem{G} Gromov, M. {\it Pseudoholomorphic curves in symplectic manifolds,}
Invent. Math. {\bf 82} (1985), no. 2, 307--347.


\bibitem{GT} Gutkin, E., Tabachnikov, S. {\it Billiards in Finsler and Minkowski geometries,}
J. Geom. Phys. 40 (2002), no. 3--4, 277--301.


\bibitem{Her} Hermann, D. {\it Non-equivalence of symplectic capacities for open
sets with restricted contact type boundary}. Pr\'epublication
d'Orsay num\'ero 32 (29/4/1998).

\bibitem{Ho} Hofer, H. {\it Symplectic capacities.} in: Geometry of
low-dimensional manifolds, 2 (Durham, 1989), 15--34, London Math.
Soc. Lect. Note Ser., {\bf 151}, Cambridge Univ. Press, 1990.

\bibitem{H1} Hofer, H. {\it On the topological properties of symplectic
maps,} Proc. Roy. Soc. Edinburgh Sect. A 115, 25--38 (1990).

\bibitem{HZ} Hofer, H., Zehnder, E. {\it Symplectic Invariants
and Hamiltonian Dynamics}, Birkh\"auser, Basel (1994).

\bibitem{Hut} Hutchings, M. {\it Quantitative embedded contact homology,} J. Differential Geom. 88 (2011), no. 2, 231--266.


\bibitem {Ku} Kuperberg, G. {\it From the Mahler conjecture to Gauss linking
integrals,} Geom. Funct. Anal., 18, no. 3, (2008),
870--892.




\bibitem{Ma} Mahler, K. {\it Ein \"Ubertragungsprinzip f\"ur konvexe
Korper,} Casopis Pyest. Mat. Fys. 68, (1939), 93--102.

\bibitem{MS} McDuff, D., Salamon, D. {\it Introduction to Symplectic Topology,}
Oxford Mathematical Monographs. The Clarendon Press, Oxford
University Press, New York, 1998.

\bibitem{ME1} Meyer, M. {\it Une caract\'erisation volumique de certains espaces
norm\'es de dimension finie,} Israel J. Math. 55 (1986), no. 3,
317--326.


\bibitem{Naz} Nazarov, F. {\it The H\"{o}rmander proof of the Bourgain-Milman theorem}, in: Geometric aspects of functional analysis, 335--343, 
Lecture Notes in Math., 2050, Springer, Heidelberg, 2012. 

\bibitem{NPRZ} Nazarov, F., Petrov, F., Ryabogin, D. and Zvavitch,
A. {\it A remark on the Mahler conjecture: local minimality of the unit cube,} Duke Math. J. 154 (2010), no. 3, 419--430. 

\bibitem{Oh} Oh, Y-G. {\it Chain level Floer theory and Hofer's geometry of the Hamiltonian diffeomorphism
group,} Asian J. Math. 6 (2002), no. 4, 579--624.

\bibitem{R1} Reisner, S. {\it Zonoids with minimal volume product,} Math.
Z. 192 (1986), no. 3, 339--346.

\bibitem{R2} Reisner, S. {\it Minimal volume-product in Banach spaces with a
1-unconditional basis,} J. London Math. Soc. 36 (1987),  no.1, 126--136.

\bibitem{SR} Saint Raymond, J. {\it Sur le volume des corps convexes
sym\'etriques,} Initiation Seminar on Analysis: G. Choquet--M.
Rogalski--J. Saint-Raymond, 20th Year: 1980/1981, Exp. No. 11, 25
pp., Publ. Math. Univ. Pierre et Marie Curie, 46, Univ. Paris VI,
Paris, 1981.

\bibitem{Sa} Santal\'o, L.A. {\it Un invariante afin para los cuerpos convexos de espacio de $n$ dimensiones,}
Portugal. Math {\bf 8} (1949) 155--161.



\bibitem{V2} Viterbo, C. {\it Symplectic topology as the geometry of
generating functions,} Math. Ann. 292, no. 4, 685--710 (1992).

\bibitem{V} Viterbo, C.
{\it Metric and isoperimetric problems in symplectic geometry,} J.
Amer. Math. Soc. 13 (2000), no. 2, 411--431.

\end{thebibliography}
\end{document}